\documentclass{amsart}
\usepackage[english]{babel}
\usepackage[utf8]{inputenc}
\usepackage{amsmath, amssymb,epic,graphicx,mathrsfs,enumerate}
\usepackage[all]{xy}
\usepackage{color}
\usepackage{comment}
\usepackage{enumitem}
\usepackage{hyperref}
\usepackage[noabbrev,capitalize]{cleveref}
\usepackage[]{frontespizio}

\usepackage{amsmath}
\usepackage{amsthm}
\usepackage{amssymb}
\usepackage{amsfonts}
\usepackage{latexsym}
\usepackage{epsfig}

\DeclareMathOperator{\aut}{Aut} 
\DeclareMathOperator{\soc}{soc}

\newcommand{\la}{\langle}
\newcommand{\ra}{\rangle}

\newtheorem{thm}{Theorem}
\newtheorem{cor}[thm]{Corollary}

 \newtheorem{lemma}[thm]{Lemma}
\newtheorem{prop}[thm]{Proposition} 
 
\newtheorem{question}[thm]{Question} \newtheorem{con}[thm]{Conjecture}

\numberwithin{equation}{section}

\renewcommand{\footnote}{\endnote}
\newcommand{\ignore}[1]{}\makeglossary

\begin{document}
	\bibliographystyle{amsplain}

\title{$p$-elements in profinite groups}
\author{Andrea Lucchini}
\address{A. Lucchini, Universit\`a di Padova, Dipartimento di Matematica ``Tullio Levi-Civita'', Via Trieste 63, 35121 Padova, Italy}
\email{lucchini@math.unipd.it}

\author{Nowras Otmen}
\address{Nowras Otmen, Universit\`a di Padova, Dipartimento di Matematica ``Tullio Levi-Civita'', Via Trieste 63, 35121 Padova, Italy}
\email{nowras.naufel@math.unipd.it}

\begin{abstract}We investigate some properties of the $p$-elements of a profinite group $G$. We prove that if $p$ is odd and the probability that a randomly chosen element of $G$ is a $p$-element is positive, then $G$ contains an open prosolvable subgroup. On the contrary, there exist groups that are not virtually prosolvable but in which the probability that a randomly chosen element of $G$ is a 2-element is arbitrarily close to 1. We prove also that if a profinite group $G$ has the property that, for every $p$-element $x$, it is positive the probability that a randomly chosen element $y$ of $G$ generates with $x$ a pro-$p$ group, then $G$ contains an open pro-$p$ subgroup.
\end{abstract}


\thanks{Project funded by the EuropeanUnion – NextGenerationEU under the National Recovery and Resilience Plan (NRRP), Mission 4 Component 2 Investment 1.1 - Call PRIN 2022 No. 104 of February 2, 2022 of Italian Ministry of University and Research; Project 2022PSTWLB (subject area: PE - Physical Sciences and Engineering) " Group Theory and Applications"}
\maketitle

\section{Introduction}

Let $G$ be a profinite group and $p$ a prime. We say that $g$ is a $p$-element of $G$ if the order of $G$, as a supernatural number, is a power of $p$. This is equivalent to say that $\langle g\rangle N/N$ is a finite $p$-group for every open normal subgroup $N$ of $G.$ We denote by $\Omega_p(G)$ the set of the $p$-elements of $G$. It is a closed subset of $G$, so we may view $P_p(G):=\mu(\Omega_p(G)),$ where $\mu$ is the normalized Haar measure on $G$, as the probability that a randomly chosen element of $G$ is a $p$-element.
One of the aims of the paper is to investigate how the structure of a profinite $G$ is influenced by the fact that $P_p(G)$ is positive. Note that if $G$ is virtually pro-$p$, then it follows immediately that $P_p(G) > 0$. However, the converse is not true. Consider $p$ and $q$ distinct primes and $n \in \mathbb{N}$ such that $p^n$ divides $q-1$. Take the semidirect product $G=\mathbb{Z}_q \rtimes \la \gamma \ra$, where $\mathbb Z_q$ is the group of the $q$-adic integers and $\gamma$ is an automorphism of $\mathbb{Z}_q$ of order $p^n$. While $G$ clearly is not virtually pro-$p$, we can see that 
$\Omega_p(G)=(G\setminus \mathbb Z_q) \cup \{1\}$
and $P_p(G) = \frac{p^n-1}{p^n}$. This example shows that $P_p(G)$ can be arbitrarily close to 1 without this guaranteeing that $G$ is virtually pro-$p$. However, we  prove that if $p$ is an odd prime and $P_p(G)$ is positive, then $G$ must contain an open prosolvable subgroup. This is indeed a corollary of a result concerning finite groups, that we  prove using some consequences of the classification of the finite simple groups.

\begin{thm}\label{podd}
		Let $p$ be an odd prime. If $N$ is a non-solvable normal subgroup of a finite group $G,$ then $$P_p(G)\leq \frac{p\cdot P_p(G/N)}{2(p-1)}.$$
\end{thm}
This result implies that, if $p$ is odd, then we can bound the number of non-abelian factors in a chief series of a finite group $G$  in terms of $P_p(G)$. More precisely, we have the following.

\begin{cor}\label{non-ab-cf}
	Let $p$ be an odd prime. Given a finite group $G,$ let $t$ be the number of non-abelian factors in a chief series of $G.$ Then
	$$t\leq \log_{\frac{2(p-1)}{p}}\left(\frac{1}{P_p(G)}\right).
	$$
\end{cor}	

If $N$ is an open normal subgroup of a profinite group $G,$ then
$P_p(G)\leq P_p(G/N).$ In particular, for every open normal subgroup $N$ of $G$, the number of non-abelian factors in a  chief series of $G/N$ is at most $\log_{\frac{2(p-1)}{p}}(\frac{1}{P_p(G)}).$ But then Corollary \ref{non-ab-cf} has the following consequences for profinite groups.

\begin{cor}\label{primocor}Let $G$ be a profinite group and $p$ an odd prime. If
	$P_p(G)> \frac{p}{2(p-1)},$ then $G$ is prosolvable.
\end{cor}

\begin{cor}\label{secondocor}Let $G$ be a profinite group and $p$ and odd  prime. If
	$P_p(G)>0,$ then $G$ is virtually prosolvable.
\end{cor}

The previous results are  \emph{not} true for $p=2$. In \cref{pro-2}, we explore the different behavior when $p=2$. In the almost simple group $M_{10}$, whose socle is the alternating group $A_6$, all the elements which are not in the socle are 2-elements. We use this property of $M_{10}$ in order to construct examples of profinite groups $G$ which contain a closed normal subgroup isomorphic to a cartesian product of infinitely many copies of $A_6$ but in which $P_2(G)$ is arbitrarily closed to 1. In Proposition \ref{anchepsl},  we show that $A_6\cong \text{PSL}(2,9)$ is not the unique simple group  $S$ admitting an automorphism $\alpha$ with the property that all the elements in the coset $\alpha S$ are 2-elements. We prove in Proposition \ref{onlypsl} that this exceptional behavior is in fact shared by all and only the simple groups $\text{PSL}(2,3^f)$, where $f$ is a 2-power. This result leads us to propose the following conjecture.

 \begin{con}\label{conj-prob-2-grps}Let $\Sigma$ the be infinite family of finite non-abelian simple groups ${\rm{PSL}}(2,3^f)$, with $f=2^a$ for a positive integer $a.$
	Suppose that a profinite group $G$ satisfies the condition $P_2(G)>0.$ Then either $G$ is virtually prosolvable or, for every $k \in \mathbb N$,  there exists an open normal subgroup $N$ of $G$ such that a composition series of $G/N$ has at least $k$ factors which belong to $\Sigma.$
\end{con}

To introduce the second part of the paper, we recall some definitions and results from \cite{aemp}. Let $\mathfrak{C}$ be a class of finite groups closed under taking subgroups, quotients and direct products. Given a profinite group $G$ and an element $x \in G$, the subset $\Omega_\mathfrak{C}(x,G)$ of elements of $G$ which generate a pro-$\mathfrak{C}$ subgroup together with $x$ is a closed subset of $G.$
We say that $G$ is $\mathfrak{C}$-positive if $\mu(\Omega_\mathfrak{C}(x,G))>0$ for every $x\in G$ and that
$G$ is  $\mathfrak{C}$-uniformly-positive if there exists a positive real number $c$ such that $\mu(\Omega_\mathfrak{C}(x,G))>c$ for every $x\in G.$ In \cite{aemp}, the authors prove that if $\mathfrak{C}$ is the class of finite solvable groups or the class of finite nilpotent groups, then $G$ is $\mathfrak{C}$-uniformly-positive
if and only if it is $\mathfrak{C}$-positive, and if the previous properties are satisfied then $\mu(\cap_{x\in G}\Omega_\mathfrak{C}(x,G))$ is positive. One can ask whether similar statements hold when $\mathfrak{C}$ is the class of the finite $p$-groups for a given prime $p.$ In this case we will prefer to use the notation $\Omega_p(g,G)$ to denote the subset $\{x \in G \colon \langle x,g \rangle \text{ is a pro-$p$ group}\}$ (here with the notation $\langle x,g\rangle$ we mean the subgroup of $G$ topologically generated by $x$ and $g,$ the smallest closed subgroup of $G$ which contains $x$ and $g$). Moreover, we will set $P_p(g,G) := \mu(\Omega_p(g,G))$. Clearly $P_p(g,G)=0$ if $g$ is not a $p$-element of $G$, so we need to adapt the definitions of
positive and uniformly-positive 
saying that $G$ is $p$-positive if $P_p(g,G)>0$ for every $p$-element of $G$ and that it is $p$-uniformly-positive if there exists a positive real number $c$ such that $P_p(g,G)>c$ for every $p$-element of $G$. 
Notice that if follows from Baer's theorem (see for example \cite [Theorem 2.12]{isa}), that the intersection
$$\bigcap_{g \in \Omega_p(G)} \Omega_p(g,G)$$
coincides with $O_p(G)$, defined as the largest normal pro-$p$ subgroup of $G$. Finally,  denote by $P_p(G,G)$ the probability that two randomly chosen elements in $G$ generate a pro-$p$ subgroup. The main result of the second part of the paper is of similar flavor to  \cite[Theorems 1.1 and 1.2]{aemp}.

\begin{thm}\label{main}
    Let $G$ be a profinite group and $p$ a prime. Then, the following are equivalent:
    \begin{enumerate}
        \item $G$ is $p$-positive;
        \item $G$ is $p$-uniformly-positive;
        \item $P_p(G,G) > 0$;
        \item $G$ is virtually pro-$p$;
        \item $\mu(\bigcap_{g \in \Omega_p(G)}\Omega_p(g,G)) = |G:O_p(G)|^{-1} > 0$.
    \end{enumerate}
\end{thm}

\subsection*{Acknowledgements} We thank  Pablo Spiga for fruitful discussions and  his precious help in the proof of Proposition \ref{anchepsl}.

\section{Probability that a random element is a $p$-element}\label{prob-p-groups}

Let $G$ be a finite group and let $P$ be a Sylow $p$-subgroup of $G.$ It follows from the Sylow theorems that $\Omega_p(G)=\cup_{g\in G}P^g$ and therefore
\begin{equation}\label{upper-bound}
    P_p(G)\leq \frac{n_p(G)|P|}{|G|}=\frac{|P|}{|N_G(P)|}.
\end{equation}

The following consequence of the classification of the finite non-abelian simple groups will play a relevant role in our considerations.

\begin{thm}\cite[Theorem 1.1]{gmn}\label{norsyl} Let $p$ be an odd prime and  $P$ a Sylow $p$-subgroup of a finite
	group $G$. If $p = 3,$ assume that $G$ has no composition factors of type ${\rm{PSL}}(2,3^f),$ $f = 3^a$ with $a \geq 1.$
	If $P=N_G(P),$ then $G$ is solvable.
\end{thm}

We will combine \ref{upper-bound} and \cref{norsyl} to bound the ratio $P_p(G)/P_p(G/N)$ when $N$ is an unsolvable normal subgroup of $G$ and $p$ is an odd prime. We need a preliminary lemma for this purpose. For any subset $X$ of $G$, we set
$\Omega_p(X)=\Omega_p(G)\cap X.$

\begin{lemma}\label{samenumber}
	Let $N$ be a normal subgroup of a finite group $G$ and $p$ a prime and assume that $G/N$ is a $p$-group. Then, for $g\in G$ and $i\in \mathbb N$ with $(p,i)=1$, the map $\rho_i: \Omega_p(gN) \to \Omega_p(g^iN)$ which sends $x$ to $x^i$ is bijective. 
\end{lemma}
\begin{proof}
	First we prove that $\rho_i$ is injective. Assume $\rho_i(x)=\rho_i(y).$ Then
	$x^i=y^i.$ In particular $|x^i|=|y^i|$ and since $x$ and $y$ are $p$-elements and $(p,i)=1,$ it follows that $|x|=|y|=p^r$ for some $r\in \mathbb N.$ Since $(i,p^r)=1,$ there exist two integers $a$ and $b$ such that $ai+bp^r=1$ and, thus, $x=x^{ai}=y^{ai}=y.$ Now, assume $z = g^in\in \Omega_p(g^iN)$ and let $|z|=p^s.$ Choose integers $a, b$ such that $ia+p^sb=1.$ Notice that
 $p^s$ is divisible by the order of $g$ modulo $N$.
  Then $z^{a}=(g^i)^an_1=g^{ia+p^sb}(g^{-p^sb}n_1)=gn_2$ for suitable $n_1, n_2\in N$ and $\rho_i(z^a)=z^{ai}=z^{ai}z^{bp^s}=z.$
\end{proof}

\begin{lemma}\label{almeno5}
	Let $p\geq 5$ be a prime.
	If $N$ is a non-solvable normal subgroup of a finite group $G$ and $g\in \Omega_p(G),$ then $$\frac{|\Omega_p(gN)|}{|N|}\leq \frac{p}{2(p-1)}.$$
\end{lemma}
\begin{proof} 
    Let $N$ be a non-solvable subgroup of a finite group $G$ and assume that $gN$ is a $p$-element of $G/N.$ Let $p^n$ be the order of $gN$, $H=\langle g, N\rangle$ and  $P$ a Sylow $p$-subgroup of $H$.  By Theorem \ref{norsyl}, we have $$|\Omega_p(H)|\leq \frac{|H||P|}{|N_H(P)|}\leq \frac{|H|}{2}.$$
	By Lemma \ref{samenumber},  $|\Omega_p(g^jN)|=|\Omega_p(gN)|$ if $(j,p)=1$ and this implies
	$$(p^n-p^{n-1})|\Omega_p(gN)|=\left|\cup_{1\leq j<p^n,(j,p)=1}\Omega_p(g^jN)\right|\leq |\Omega_p(H)|\leq \frac{|H|}{2}=\frac{p^n|N|}{2},
	$$
	therefore,
	\[|\Omega_p(gN)|\leq \frac{p|N|}{2(p-1)}.\qedhere\]
\end{proof}

We are going to prove that the statement of the previous lemma is true even if $p=3.$ But in that case the argument used in the proof of Lemma \ref{almeno5} does not work, since there exist unsolvable groups whose Sylow 3-subgroups are self-normalizing. The following precise classification of the finite almost simple groups with this property is required.

\begin{lemma}\cite[Lemma 5]{vdovin}\label{vdo3} Assume that $G$ is an almost simple group. If the Sylow 3-subgroups of $G$ are self-normalizing, then $S=\soc(G)={\rm{PSL}}(2,3^{3^n})$
	and $G=S\rtimes \langle\phi\rangle$, with $\phi$ a field automorphism of $S$ of order $3^n$.
	\end{lemma}
	
To deal with the almost simple groups described in the previous Lemma we need to estimate the number of 3-elements in the coset $\phi\text{PSL}(2,3^{3^n}).$

\begin{lemma}\label{l23} Let $S=\soc(G)={\rm{PSL}}(2,3^{3^n})$
	and $G=S\rtimes \langle\phi\rangle$, with $\phi$ the field automorphism of $S$ of order $3^n$. Then
$$\sup_{n\in \mathbb N}\frac{|\Omega_3(\phi S)|}{|S|}=\frac{3}{4}.$$
\end{lemma}
\begin{proof} It can be easily checked that $C:=C_S(\phi)=\text{PSL}(2,3)\cong A_4$. In particular, there exists $z\in C$ of order 2. Let
$y=\phi z.$ Then $y^S=\{y^s\mid s\in S\}\subseteq \phi S$ 
and $|y^S|=|S|/|C_S(y)|=|S|/4$ (for $C_S(y)=C_S(\phi)\cap C_S(z) \cong C_2 \times C_2$). Since $|y|=2\cdot 3^n$, we have that $y^{S}\cap \Omega_3(G)=\emptyset,$ so $|\Omega_3(\phi S)\leq |S|-|y^S|\leq \frac{3|S|}4.$ Finally, it is straightforward to verify using \cite{atlas} that $|\Omega_3(\phi S)|/|S|=3/4$ in the particular case when $n=1.$
\end{proof}

\begin{cor}\label{corx3}
	If $S$ is a finite non-abelian simple group and $g\in \aut(S),$ then $$\frac{|\Omega_3(g S)|}{|S|}\leq\frac{3}{4}.$$
\end{cor}
\begin{proof}Let $H=\langle g,S\rangle \leq \aut(S)$ and let
	$P$ be a Sylow 3-subgroup of $G.$ If $N_H(P)>P,$ then, arguing as in the proof of \cref{almeno5}, we deduce that 
	$|\Omega_3(gS)|\leq 3|S|/4.$ If $N_H(P)=P,$ then, by  \cref{vdo3}, $S={\rm{PSL}}(2,3^{3^n})$
	and $H=S\rtimes \langle\phi\rangle$, with $\phi$ a field automorphism of $S$ of order $3^n$. Thus $gS=\phi^iS$ for some positive integer $i$ coprime with 3, and, by Lemmas \ref{samenumber} and \ref{l23},
	$|\Omega_3(gS)|=|\Omega_3(\phi S)|\leq 3|S|/4.$
	\end{proof}

\begin{lemma}\label{tre}
 Assume that $N$ is a normal subgroup of $G$ and let $g\in \Omega_3(G).$ 
	If $N$ is a non-solvable normal subgroup of a finite group $G,$ then $$\frac{|\Omega_3(gN)|}{|N|}\leq \frac{3}{4}.$$
\end{lemma}	
\begin{proof}
	It is not restrictive to assume that $N$ is a minimal normal subgroup of $G$. So $N\cong S^t$ with $t\in \mathbb{N}$ and $S$ a finite non-abelian simple group. 
 Let $\rho: G \to \aut(N)$ be the conjugation action. If $gn$ is a 3-element, then
	$(gn)^\rho$ is a 3-element, so it suffices to prove that $N$ contains at most $3|N|/4$ elements $n$ such that $(gn)^\rho=g^\rho n^\rho$ is a 3-element. For this purpose, we may assume that
	$\ker \rho=1$ (or equivalently that $N$ is the unique minimal normal subgroup of $G$). Thus, we may identify $G$ with a subgroup of the wreath product $\aut S \wr S_t.$ \\
    \indent Let $g=\sigma(h_1,\dots,h_t)$, with $h_1,\dots,h_t \in \aut S$ and $\sigma \in S_t$, be a $3$-element. We have to prove that the number of elements $n=(s_1,\dots,s_t)\in S^t$ such that $gn=\sigma(h_1s_1,\dots,h_ts_t)$ is a $3$-element has size at most $3|S|^t/4.$ Since $\sigma$ is a 3-element, it is the product of disjoint cycles of length a power of 3. 
    It suffices to prove the statement in the particular case when $\sigma$ is a cycle of length $t$. Up to conjugacy in $\aut(N),$ we may then assume $g=\sigma(h,1,\dots,1)$ with $\sigma=(1,2,\dots,t)$ and $t=3^a$. Since $g$ is a 3-element,
	$g^t=(h,\dots,h)\in (\aut S)^t$ is a 3-element, hence $h$ is a 3-element in $\aut S.$ Moreover, if $n=(s_1,\dots,s_t),$ then
	$$(gn)^t=(s_2s_3\cdots s_ths_1,s_3s_4\cdots hs_1s_2,\dots,hs_1s_2\dots s_{t-1}s_t),
	$$
	so $gn$ is a 3-element if, and only if, $hs_1s_2\dots s_{t-1}s_t$ is a 3-element. Since this element belongs to the almost simple group $\langle h,S \rangle$, then by  \cref{corx3}, the number of choices for the product $s_1\cdots s_t$ such that $hs_1\cdots s_t$ is a 3-element is at most $3|S|/4$ and this implies that the number of choices for $n$ is at most $3|N|/4.$
\end{proof}

\begin{proof}[Proof of Theorem \ref{podd}] Since $\Omega_p(G)=\cup_{gN\in \Omega_p(G/N)}\Omega_p(gN),$ it follows from Lemmas \ref{almeno5} and \ref{tre} that
	$$\begin{aligned}P_p(G)&=\frac{|\Omega_p(G)|}{|G|}\leq \sum_{gN\in \Omega_p(G/N)}\frac{|\Omega_p(gN)|}{|G|}
\leq\sum_{gN\in \Omega_p(G/N)}\frac{p|N|}{2(p-1)|G|}\\
&\leq \frac{p}{2(p-1)}\frac{|\Omega_p(G/N)||N|}{|G|}\leq \frac{p P_p(G/N)}{2(p-1)}.\qedhere
\end{aligned}$$
\end{proof}


\section{Exceptional behavior for $p=2$}\label{pro-2}

In this section, we give examples to show why the results of \cref{prob-p-groups} do not hold for $p=2$ and we highlight that our counterexamples rely on the exceptional behaviour of the set of the 2-elements in the Mathieu group $M_{10}$.

\indent Consider $X=M_{10}$. It is an almost simple group with $S = \soc(X) = A_6$ and $X/S \cong C_2.$ It can be deduced from the information given in \cite{atlas}, that $X$ contains  496 2-elements, 136 of which are in $S$. 
In particular all the elements of $X\setminus S$ are 2-elements.

Given now a positive integer $t,$ consider the following subgroup of the direct power $X^t:$
$$X_t=\{(x_1,\dots,x_t)\in X^t \mid x_1\equiv \cdots \equiv x_t \mod S\}.$$
Notice that $\soc(X_t)=S^t$ and $X_t/\soc(X_t)\cong C_2.$
Let $\alpha=(a,\dots,a)\in X^t,$ with $X=\langle a, S\rangle.$ Then $X_t$ is the disjoint union of the two cosets $\alpha S^t$ and $S^t$. All the elements of $\alpha S^t$
have order dividing 8, while $S^t$ contains precisely $136^t$ 2-elements.  It follows
$$P_2(G_t)=\frac{(136)^t+(360)^t}{2(360)^t}=\frac{1}{2}+\frac{1}{2}\frac{(136)^t}{(360)^t}
.$$
It $t_2\geq t_1$, then $G_{t_1}$ is an epimorphic image of $G_{t_2}$ so we may consider the profinite group
$$G:=\varprojlim_{t \in \mathbb N} G_t.$$
On one hand, $G$ is not virtually prosolvable because every open subgroup contains an infinite product of copies of $A_6$ and, on the other, we have 
$$P_2(G)=\lim_{t\to \infty}P_2(G_t)=\lim_{t\to \infty}
\frac{1}{2}\left(1+\frac{(136)^t}{(360)^t}\right)=\frac{1}{2}.$$

We build now a more complicated example to show that we can find a non-virtually prosolvable profinite group $G$ with $P_2(G)$ arbitrarily close to 1. Let $\sigma=(1,\dots,n)\in S_n$ with $n=2^t$. Choose again $a\in X\setminus S$ and take the element $g=(a,1,\dots,1)\sigma$ in the wreath product $X\wr \langle \sigma \rangle$  and let $Y_t=N\langle g\rangle$ be the subgroup of
$X\wr \langle \sigma \rangle$ generated by the base subgroup $N=S^n$ and $g.$  Let $\pi: Y_t\to \langle \sigma \rangle$ be the restriction to $Y_t$ of the projection $X\wr \langle \sigma \rangle \to \langle \sigma \rangle.$ If $z\in Y_t\setminus N$, then $|\pi(z)|=2^r$ with $r\leq t$ and $z^{2^r}=(x_1,\dots,x_n)$ with
$x_i \in X\setminus S$ for $1\leq i\leq n.$ In particular, $(x_1,\dots,x_n)$ is a 2-element and therefore $z$ is a 2-element. 
Consequently, 
$$P_2(Y_t)\geq \frac {|Y_t|-|N|}{|Y_t|}=\frac{2^{t+1}-1}{2^{t+1}}.$$
Now, given $u\in \mathbb N,$ let $X_u=\prod_{t\geq u}Y_{2^t}$. Then
$$P_2(X_u)\geq \prod_{t\geq u}
\left (1-\frac{1}{2^{t+1}}\right).$$
Since the infinite product $\prod_{t\in \mathbb N}
\left (1-\frac{1}{2^t}\right)$ converges to a non-zero value (because the series $\sum_t \frac{1}{2^t}$ converges), for every $0< c < 1,$ there exists $u\in \mathbb N$ such that $P_2(X_u)\geq c.$ 

The previous examples rely on the fact that $A_6\cong \text{PSL}(2,9)$ has an automorphism $g$ of order 4 with the property that all the elements in the coset $gA_6$ are 2-elements. This is a particular instance of a more general situation.

\begin{prop}\label{anchepsl}
	Let $S={\rm{PSL}}(2,3^f)$, with $f\neq 1$.
	Let $g\in {\rm{PGL}}(2,3^f)\setminus {\rm{PSL}}(2,3^f)$ and $\phi$ be the Frobenius automorphism of $S$ of order $f$. Then $\alpha=g\phi \in \aut(S)$ and all the elements of the cosets $\alpha S$ have order dividing $4f.$ In particular, if $f$ is a 2-power, then $\alpha S$ contains only 2-elements.
\end{prop}

\begin{proof} Let $\mathbb{F}$ be the algebraic closure  of the  field with 3 elements, consider $X=\text{PSL}(2,\mathbb{F})$ and let $\phi$ be the Frobenius endomorphism of $X$ induced by the automorphism of $\mathbb{F}$ which maps $a$ to $a^3.$ For any $f \in \mathbb{N}$, we have that $X_{\phi^f} = C_X(\phi^f)=\text{PGL}(2,3^f)$. If $f >1$, then $X_{\phi^f}$ is $\phi$-stable and $\phi$ restricts to an automorphism $\tilde{\phi}$ of $X_{\phi^f}$ of order $f$. Let $G_f = \text{PGL}(2,3^f) \rtimes \langle\tilde{\phi}\rangle$. Choose $a\in X$ such that $g=aa^{-\phi^{-1}}$.
	By \cite[Theorem 3.1.4]{har}, we have the Shintani map
\[\begin{aligned}\sigma: \{(g\tilde \phi)^{X_{\phi^f}}\mid g \in X_{\phi^f}\}&\to \{(x)^{X_{\phi}}\mid x \in X_{\phi}\} \\
(g\tilde \phi)^{X_{\phi^f}}&\mapsto (a^{-1}(g\tilde \phi)^ea)^{X_{\phi}}.\end{aligned}\] This is a well-defined bijection between the $X_{\phi^f}$-conjugacy classes  in $G_f$ that are contained in the coset $X_{\phi^f}\tilde{\phi}$ and the  $X_{\phi}$-conjugacy classes in $X_\phi$. Since $\langle \text{PSL}(2,3^f), \tilde{\phi} \rangle \unlhd G_f$, then $\sigma$ restricts to a bijection
\[\{(g\tilde \phi)^{X_{\phi^f}}\mid g \in \text{PSL}(2,3^f)\} \to \{(x)^{X_{\phi}}\mid x \in \text{PSL}(2,3)\}\]
by \cite[Corollary 3.2.3]{har}. Consequently, it also restricts to a bijection
\[\{(g\tilde \phi)^{X_{\phi^f}}\mid g \in X_{\phi^f}\setminus \text{PSL}(2,3^f)\} \to \{(x)^{X_{\phi}}\mid x \in X_\phi \setminus \text{PSL}(2,3)\}.\]
From the isomorphisms $X_\phi = \text{PGL}(2,3) \cong S_4$ and $\text{PSL}(2,3) \cong A_4$, it follows that $x^4 = 1$ for every element  $x \in \text{PGL}(2,3) \setminus \text{PSL}(2,3)$. In particular, if $g \in \text{PGL}(2,3^f) \setminus \text{PSL}(2,3^f)$, then $(g\tilde{\phi})^{4f}=1$. If $f$ is a power of 2, then $g\tilde{\phi}$ is a 2-element.
\end{proof}

 We now prove that if $S$ is a finite non-abelian simple group and $g$ is an automorphism of $S$ with the property that all the elements of $gS$ are $p$-elements for some prime $p,$ then $p=2$ and the pair $(S,g)$ is one of those described in the previous proposition.

\begin{lemma}
	Let $G=\langle g, S\rangle$ be an almost simple group with $S=\soc(G)$ and $|G/S|=p^n$ for some prime $p.$ If $gS \subseteq \Omega_p(G),$ then $p=2$ and $N_G(P)=P$ for any Sylow $p$-subgroup $P$ of $G.$
\end{lemma}
\begin{proof}
	As in the proof of Proposition \ref{almeno5},
	$$|\Omega_p(gS)|\leq \frac{p|S|}{(p-1)|N_G(P):P|}.$$
	In particular, if $gS \subseteq \Omega_p(G),$ then
	$|N_G(P):P|\leq p/(p-1).$
	If $p$ is odd, then $p/(p-1)<2$, so $|N_G(P):P|<2,$ and therefore
	$N_G(P)=P$, but if follows from Theorem \ref{norsyl} and \cref{corx3} that this case cannot occur. So $p=2$ and
	$|N_G(P):P|\leq 2.$ Since 2 does not divide $|N_G(P):P|,$ it follows that $N_G(P)=P.$
\end{proof}

\begin{prop}\label{onlypsl}
    Let $S$ be a finite non-abelian simple group and $g \in \aut(S)$. If $gS$ consists only of 2-elements, then {\normalfont{$S \cong \text{PSL}(2,3^f)$}} with $f$ a 2-power and $g=\phi^i x$ with	 $x\in {\rm{PGL}}(2,3^f)\setminus {\rm{PSL}}(2,3^f)$, where $\phi$  is the Frobenius automorphism of $S$ of order $f$ and $i$ is odd.
\end{prop}
\begin{proof}
By assumption, $g$ is a 2-element. Since $S$ is not a 2-group, we may exclude that $g$ is an inner automorphism of $S$. Using the information in \cite{atlas}, it can be seen that $S$ cannot be a sporadic simple group. Moreover if $S=A_n$ and $n\neq 6$, then $\aut(A_n)=S_n$ and $gA_n= S_n \setminus A_n$ contains $(1,2,3)(4,5).$ Moreover it follows from \cite[Lemma 6]{ln} and its proof, that if $S$ is a simple group of Lie type defined over a field of characteristic $p$, then either $S=\text{PSL}(2,q)$ or $gS$ contains an element of order divisible by $p$. Thus if
$gS\subseteq \Omega_2(G),$ then either $S$ is defined over a field of characteristic 2 or $S=\text{PSL}(2,q).$

In the first case, $\text{Inndiag}(S)/S$ has odd order, so we may assume that $g$ is the product of a graph automorphism $\tau$ and a Frobenius automorphism $\phi$. However, it can be easily checked that $C_S(\langle\phi,\tau\rangle)$ contains at least one non-trivial element $s$ of odd order, which would imply that $gs$ is an element of the coset $gS$ which is not a 2-element.

So we may assume $S=\text{PSL}(2,q)$ with $q$ odd. We may exclude again that $g$ is a Frobenius automorphism because in this case $C_S(g)$ contains an element of odd order; moreover $g$ cannot be inner-diagonal, because in this case $gS$ contains an element of order $q-1$ and one of order $q+1$ (\cite[Kapitel II, Hauptsatz 8.27]{hup}), and either $q-1$ or $q+1$ is divisible by an odd prime (unless $q=3$, but in this case $\text{PSL}(2,3)$ is not simple). We view $\text{PSL}(2,q)$ as the factor group $\text{SL}(2,q)/Z,$
where $Z$ is the set of the scalar matrices in  $\text{SL}(2,q)$. We deduce from the previous considerations that $g=\phi\alpha$ where $\alpha$ is a diagonal non-inner automorphism of $S$ and $\phi$ is a Frobenius automorphism. 
By  \cref{samenumber} we may assume that there exists a prime-power $r$ such that $\phi$ is 
the Frobenius automorphism  induced by the automorphism
$x \mapsto x^r$ of the field of order $q=r^t$. Since $t$ divides $|g|,$ $t$ is a 2-power. In particular, denoting by $\tilde y$ the automorphism induced by the conjugation with $y\in \text{GL}(2,q),$ we have that
$$gS = \{\phi\tilde y \mid  y\in \text{GL}(2,q) \text { and }\det (y) \notin \mathbb{F}_q^2\}.$$ 
Let $a \in \mathbb{F}_q$ be an element of order $q-1$. First, consider the matrix
\[y =
    \begin{pmatrix}
        a & 0 \\
        0 & 1
    \end{pmatrix}.
\]
Then
\[(\phi y)^t = 
    \begin{pmatrix}
        a^{\left( \frac{q-1}{r-1}\right)} & 0 \\
        0 & 1
    \end{pmatrix}.\]
    Since $\phi \tilde y$ is a 2-element and $t$ is a 2-power,
    $(\phi \tilde y)^t$ is also a 2-element
    and thus $r-1$ must be a 2-power. Secondly, we consider the matrix
\[y=
    \begin{pmatrix}
        0 & -a \\ 1 & 0
    \end{pmatrix}.   
\]
Then
\[(\phi y)^t = 
    \begin{pmatrix}
        a^{\left( \frac{q-1}{r^2-1}\right)} & 0 \\
        0 & a^{r\left( \frac{q-1}{r^2-1}\right)} 
    \end{pmatrix}\]
and from the fact that $(\phi \tilde y)^t$ is a 2-element, it follows that $r+1$ is also a 2-power. We have so proved that $r-1$ and $r+1$ are both 2-powers, and this is possible only if $r=3$ and $q=3^t$ with $t$ a 2-power.
\end{proof}

Notice that the argument in the proof of Lemma \ref{tre} can be adapted to prove the following result:

\begin{lemma}\label{bgam}
Assume that $N$ is a minimal normal subgroup of a finite group $G$ and let $S$ be a composition factor of $N$. Assume that there exists a constant $\gamma_S<1$ with the property that $|\Omega_2(gS)|\leq \gamma_S|S|$ for every $g\in \aut(S)$.
Then $|\Omega_2(gN)|\leq \gamma_S|N|$ for every $g\in \Omega_2(G).$
\end{lemma}

Given a positive real number $\gamma<1$, let $\Sigma_\gamma$ be the set of the non-abelian simple groups $S$ with the property that $|\Omega_2(gS)|\leq \gamma|S|$ for every $g\in \aut(S).$ As a consequence of \cref{bgam}, the following holds:

\begin{thm}
Let $G$ be a profinite group and suppose that there exist $\gamma<1$ and $t\in \mathbb N$ with the property that for  every open normal subgroup $N$ of $G,$ a composition series of $G/N$ contains at most $t$ non-abelian factors which are not in $\Sigma_\gamma.$ If $P_2(G)>0,$ then $G$ is virtually prosolvable.
\end{thm}

This motivates the following question:

\begin{question}\label{quest}
Does there exist a constant $\gamma < 1$ with the property that either $S\in \Sigma_\gamma$ or {\normalfont{$S \cong \text{PSL}(2,3^f)$}} with $f$ a 2-power?
\end{question}

If the answer to the previous question were affirmative, then Conjecture \ref{conj-prob-2-grps} would be true.
The following small result implies that in order to answer Question \ref{quest}, one has only to consider the simple groups of Lie type.


\begin{lemma}\label{anchealt} 
	There exists $\gamma<1$ such that $A_n \in \Sigma_\gamma$ for every $n\neq 6.$
\end{lemma}

\begin{proof}
It follows from \cite[(4.1)]{et} that the probability that an element of $S_n$ has order coprime with 3 is at most $3\cdot \text{exp}\left(-\frac{\log(n/3)}{3}\right).$ So the proportion of 2-elements in $S_n$ tends to zero as $n$ goes to infinity and the same is true for $A_n$ and $S_n \setminus A_n$.
 \end{proof}


\section{$p$-positive profinite groups}\label{p-pos-prof-grps}

In this section our aim is to prove Theorem \ref{main}. The first observation is that if a profinite group $G$ is $p$-positive then $P_p(G)=P_p(1,G)>0.$ If $p$ is odd, we immediately deduce from Corollary \ref{secondocor} that
$G$ is virtually prosolvable. We are going to prove that this is true also for $p=2$, but this requires some additional work which uses again some consequences of the classification of the finite non-abelian simple groups. In particular a key role is played by the following lemma.

\begin{lemma}\label{univbound}There exists a positive real number $\epsilon$ with the following property. For every finite group $G$, every non-abelian minimal normal subgroup $N$ of $G$ and every $x\in \Omega_2(G)$ 
	there exists $z\in \Omega_2(Nx)$ such that $$P_2(z,G)\leq (1-\varepsilon) P_2(xN,G/N).$$
\end{lemma}
\begin{proof}Let $\Sigma$ be the set of the elements $y\in G$ such that $\langle x,y\rangle N/N$ is a 2-group. For any $z\in \Omega_2(xN)$ and $y\in \Sigma$, define $\Lambda_{z,y}=\{n\in N\mid \langle z,yn\rangle \text { is a 2-group}\}.$
If we prove that there exists a $z \in \Omega_2(xN)$ such that $|\Lambda_{z,y}| \leq (1-\varepsilon)|N|$ for every $y\in \Sigma,$ then the statement follows from the equality
\[\Omega_2(z,G) = \bigcup_{y\in \Sigma} \Lambda_{z,y}.\]
Considering the map $\gamma: G\to \aut N$ induced by the conjugation action, we have that
$|\Lambda_{z,y}|\leq |\{n\in N\mid \langle z^\gamma,(yn)^\gamma \rangle \text { is a 2-group}\}|.$ So, we may assume that $\gamma$ is injective. In this case, $N=S^m$ with $S$ a finite non-abelian simple group and $G\leq \aut(S)\wr S_m.$ We have
$x=(h_1,\dots,h_m)\sigma$ with $h_i\in \aut(S)$ and $\sigma\in S_m.$ We write $\sigma=\sigma_1 \cdots \sigma_r$ as a product of disjoint cycles. Since $x\in \Omega_2(G)$, then $|\sigma_i|=2^{t_i}$ for some non-negative integer $t_i.$ We may assume
$2^{t_1}\geq \dots \geq 2^{t_r}$ and $\sigma_1=(1,\dots,2^{t_1}).$
We have $x^{2^{t_1}}=(k_1,\dots,k_m)$ with $k_i\in \aut(S).$ In particular, $k_1=h_1\cdots h_{2^{t_1}}.$ If $k_1\neq 1,$ we  take $z=x.$ If $k_1=1,$ we consider $z=(s,1,\dots,1)x.$ Notice that
$z^{2^{t_1}}=(s,\dots,s,k_{2^{t_1}+1},\dots,k_m)$, therefore, if we choose $1\neq s\in \Omega_2(S)$, then $z\in \Omega_2(xN)$  and $z^{2^{t_1}}=(u_1,\dots,u_m)\in (\aut(S))^m$ with
$u_1\neq 1.$ 

Now, assume $y=(v_1,\dots,v_m)\tau$ and write $\tau=\tau_1\cdots \tau_l$ as product of disjoint cycles, with $\tau_1=(1,j_2,\dots,j_d)$. Let $n=(s_1,\dots,s_m)\in N$ and consider the subgroups $S_1=\{(s,1,\dots,1)\mid s\in S\}\cong S$ and $X=N_G(S_1)$, and let $\nu: X\to \aut(S_1)$ be the homomorphism which maps $x$ to the automorphism of $S_1$ induced by the conjugation with $x$. We have that $z^{2^{t_1}}$ and $(ny)^d$ belong to $X$ with $\nu(z^{2^{t_1}})=u_1$ and $\nu((ny)^d)=s_1v_1s_{j_2}v_{j_2}\cdots s_{j_d}v_{j_d}.$
Choose $s_2,\dots,s_m$ arbitrarily. By \cite[Corollary 1.2]{fgg}, with probability at least $\epsilon$, a randomly chosen element $s_1\in S$ has the property that $\langle u_1, s_1v_1s_{j_2}v_{j_2}\cdots s_{j_d}v_{j_d}\rangle$ is nonsolvable, and consequently $\langle z, ny\rangle$ is not a 2-group. In particular, $|\Lambda_{z,y}|\leq (1-\varepsilon)|N|.$
\end{proof}

\begin{cor}\label{virtprosolv-2}
  If a profinite group $G$ is $p$-positive, then $G$ is virtually prosolvable.
\end{cor}

\begin{proof} As we noticed before, if $G$ is $p$-positive then $P_p(G)=P_p(1,G)>0.$ If $p$ is odd, it follows  from Corollary \ref{secondocor} that
	$G$ is virtually prosolvable.
	So it suffices to prove that if
a profinite group $G$ is not virtually prosolvable, then there exists $g\in G$ such that $P_2(g,G)=0.$

    Assume that $G$ is not virtually prosolvable. Then there exists a descending chain of open normal subgroups of $G$
	$$G\geq N_1>M_1\geq N_2>M_2\geq N_3>M_3\dots$$
	such that $N_t/M_t$ is a non-abelian minimal normal subgroup of
	$G/M_t$ for every $t\in \mathbb N.$ Since $P_2(G)\leq P_2(G/L)$ for every closed normal $L$ subgroup of $G$, it is not restrictive to assume $\cap_{t\in \mathbb N}M_t=1.$  In particular, we may identify $G\cong \varprojlim_{i \in \mathbb N}G/M_i$ with the subgroup of the cartesian product $\prod_{i\in \mathbb N}G/M_i$ consisting of the elements $(g_iM_i)_{i\in \mathbb N}$ with $g_iM_j=g_jM_j$ whenever $j\leq i.$
  
   Consider $g_1 \in G$ such that $g_1M_1 \in \Omega_2(G/M_1)$. We may assume that $g_1 \in \Omega_2(G)$. In this case, $g_1 N_2\in \Omega_2(G/N_2)$ and, by applying \cref{univbound}, we can find $g_2 \in g_1 N_2$ such that
    \[P_2(g_2M_2,G/M_2) \leq (1-\varepsilon)P_2(g_1 N_2, G/N_2) \leq (1-\varepsilon)P_2(g_1M_1, G/M_1).\]
    By construction, it follows also that $g_2M_1 = g_1M_1$. Inductively, we obtain a sequence $(g_tM_t)_{t \in \mathbb{N}}$ of elements of $G$ such that $g=(g_tM_t)_{t\in \mathbb N} \in \Omega_2(G)$ and
    \[P_2(g,G)=\inf_{t\in \mathbb N} P_2(g_{t+1}M_{t+1},G/M_{t+1})\leq \inf_{t\in \mathbb N}(1-\varepsilon)^tP_2(g_1M_1,G/M_1)=0.
	\qedhere\]
\end{proof}


\begin{lemma}\label{gxfinite}
    Let $G$ be a finite group, $N$ an abelian normal subgroup of $G$ and $g$ a $p$-element of $G.$ If $p$ does not divide $|N|,$ then
$$\frac{P_p(g,G)}{P_p(gN,G/N)}=
\frac{1}{|N||\Omega_p(gN,G/N)|}\sum_{xN \in \Omega_p(gN,G/N)}\frac{|C_N(g)|}{|C_N(g)\cap C_N(x)|}.$$
In particular, if $C_N(g)\neq N,$ then
$P_p(g,G)\leq P_p(gN,G/N)/2.$
\end{lemma}
\begin{proof}
    Take $xN \in \Omega_p(gN,G/N).$ There exists a Sylow $p$-subgroup $P$ of $G$ such that $\langle g, x\rangle \leq PN$ and it is not restrictive to assume $\langle g, x\rangle \leq P.$ Suppose that $n\in N$ is such that $\langle g, xn\rangle$ is a $p$-group. Then there exists $y \in N$ such that $\langle g, xn\rangle \leq P^y.$ Since $P\cap N=1,$ it follows that $g^y=g$ and $xn=x^y$ and therefore the number of $n\in N$ such that $\langle g,xn\rangle$ is a $p$-group coincides with  the number $|C_N(g):C_N(g)\cap C_N(x)|$ of  conjugates of $x$ with elements of $C_N(g).$  Now, from
    \[\Omega_p(g,G) = \bigsqcup_{xN \in \Omega_p(gN,G/N)} \Omega_p(g,G)\cap xN\]
    we get the desired results.
\end{proof}

To state the following result we need a definition. Let $G$ be a finite group, $p$ a prime and $g$ a $p$-element of $G$. We denote by $\omega_p(g,G)$ the number of abelian factors $X/Y$ in a chief series of $G$ such that $p$ does not divide $|X/Y|$ and $g$ does not centralize $X/Y.$ 

\begin{cor}\label{bbb}
	If $G$ is a finite group and $g$ is a $p$-element of $G$, then 
	$$P_p(g,G)\leq \frac{1}{2^{\omega_p(g,G)}}.$$
\end{cor}

\begin{proof}
    Take a chief series
   $1 = N_t \unlhd N_{t-1} \unlhd \cdots \unlhd N_0 = G$
    of $G$ and apply \cref{gxfinite} repeatedly to the abelian $p'$-chief factors which are not centralized by $g$.
\end{proof}

Let us recall a result concerning finite solvable groups.

\begin{thm}\label{thm_dh} Let $G$ be a finite solvable group and $p$ a prime divisor of $G$. Then a $p$-element $g$ of $G$ centralizes all the $p'$-chief factors of $G$ if and only if $g \in O_p(G).$
\end{thm}
\begin{proof}
Let $q$ be a prime divisor of $|G|$ and let $C_q(G)$ be the intersection of the centralizers of the chief factors of $G$ whose order is divisible by $q
.$ By \cite[Theorem 13.8]{dh}, we have that $C_q(G)=O_{q^\prime,q}(G).$ So it is sufficient to prove that a Sylow $p$-subgroup $P$ of $C=\cap_{q\neq p}O_{q^\prime,q}(G)$ is normal in $C$ and coincides with $O_p(G).$
It follows from \cite[Theorem 13.4]{dh}
that $C$ is $q$-nilpotent for every $q\neq p.$ If $K_q$ is the normal $q$-complement in $C$, then $P=\cap_{q\neq p}K_q$.  Since $C$ is a normal subgroup of $G$, it follows that $P\leq O_p(G).$ On the other hand $O_p(G)\leq O_{q^\prime,q}(G)$ for every $q\neq p,$ so $O_p(G)\leq P.$
\end{proof}

	For a profinite group $G$ and a $p$-element $g \in G$, we say that $g$ \emph{centralizes  all the abelian chief factors of $G$ of order coprime to $p$} (henceforth, $p'$-chief factors) 
 if $g$ centralizes all abelian chief factors of $G/N$ for every open normal subgroup $N$ of $G.$ We will say that $g$ \emph{centralizes  almost all the abelian chief factors of $G$ of order coprime to $p$} if
		there exists a non-negative integer $t$ such that, for every open normal subgroup $N$ of $G$,  a chief series of $G/N$ contains at most $t$ abelian factors of order coprime to $p$ that are not centralized by $g.$ With these definitions, we may formulate Theorem \ref{thm_dh} in the context of profinite groups.
\begin{cor}\label{opg}
	Let $G$ be a prosolvable group and let $g$ be a $p$-element of $G$ which centralizes all the chief factors of $G$ whose  orders are coprime with $p.$ Then $g \in O_p(G).$
\end{cor}

\begin{thm}\label{42}
	Let $G$ be a  virtually prosolvable profinite group and let $g$ be a $p$-element of $G$. 
	If $P_p(g,G)>0,$ then there is an open subgroup $H$ of $G$ such that  $g \in O_p(H).$ In particular, $g$ has finite order modulo $O_p(G).$ 
\end{thm}
	
\begin{proof}
	Suppose $P_p(g,G)>\eta$. By \cref{bbb}, $\omega_p(gN,G/N) \leq - \log_2 \eta$ for every open normal subgroup $N$ of $G$ and we deduce from this that $g$ centralizes almost all abelian $p'$-chief factors of $G$. In particular, $K$ contains an open normal subgroup $M$ of  $G$ such that $g$ centralizes all the $p^\prime$-chief factors of $M$. It follows from \cref{opg} that $g \in O_p(M\langle g \rangle)$.
  Moreover, $\langle g\rangle \cap M \leq O_p(M)\leq O_p(G)$ implies $|gO_p(G)|\leq |G/M|.$
\end{proof}

A direct observation, which is implicit in the proof of \cref{42}, is that if a $p$-element $g$ of a profinite group $G$ centralizes all the $p'$-chief factors of an open normal prosolvable subgroup $M$ of $G$ which contains $g$, then $g\in O_p(M)$.

\begin{cor}\label{sylowfg}
Let $G$ be a $p$-positive profinite group. If the $p$-Sylow subgroups of $G$ are topologically finitely generated, then $G$ is virtually pro-$p.$
\end{cor}

\begin{proof}
	Let $P$  be a $p$-Sylow subgroup of $G$
	and $g_1,\dots,g_d$ a set of topological generators for $P.$ By \cref{virtprosolv-2},
	$G$ is virtually prosolvable, so it follows from \cref{42}, that, for each $1\leq i\leq d$, there exists an open normal subgroup $M_i$ of $G$ such that $g_i$ centralizes all the $p'$-chief factors of $M_i$. Let $M=M_1\cap \dots \cap M_d.$ By construction, every $p'$-chief factor of $M$ is centralized by every element of $P$. If we consider any other $p$-element $z$, there exist $x \in P$ and $y \in G$ such that $z=x^y$. If $H/L$ is a $p'$-chief factor of $M$ and $h \in H$, then
    \[[z,h]=[x^y,h]=[x,h^{y^{-1}}]^y \in L,\] 
    so every $p$-element of $G$ centralizes every $p'$-chief factor of $M$. In particular, it follows from Corollary \ref{opg} that the $p$-Sylow subgroup $X=P\cap M$ of $M$ is normal in $G.$ We may assume $X=1$, which implies that $P$ is finite. Moreover, $P\leq O_p(PM)$ by the observation after \cref{42} and therefore $[P,M]=1.$ Thus, $N_G(P)$ contains $M$ and, consequently, $\Omega_p(G)$ is finite. Since $P_p(1,G)=\mu(\Omega_p(G))$ is positive, we conclude that $G$ is finite. 
\end{proof}

\begin{thm}\label{virtprop}
    Let $G$ be a $p$-positive profinite group. Then $G$ is virtually pro-$p$.
\end{thm}
\begin{proof}
    We may assume $O_p(G)=1$, since $G$ is virtually pro-$p$ if, and only if, $G/O_p(G)$ is virtually pro-$p$. First, suppose that $G$ is countably based. By \cref{virtprosolv-2}, $G$ contains  an open normal prosolvable subgroup $K$. Thus, there exists a descending chain  $\{N_i\}_{i\in \mathbb N}$ of open normal subgroups of $G$ such that $N_0=G,$ $N_j=K$ for some $j\in \mathbb N,$ $\cap_{i\in \mathbb N}N_i=1$ and $N_i/N_{i+1}$ is a chief factor of $G/N_{i+1}$ for every $i\in \mathbb N.$    
  Let $P$ be a  $p$-Sylow subgroup of $G$. Moreover
   let $\Delta$ be the set of  $n \in \mathbb{N}$ such that $u\geq j$ and $N_{u}/N_{u+1}$ is a $p'$-chief factor and let $\Lambda$ be the set of cofinite subsets of $\Delta$. For every $\lambda \in \Lambda$, we have that $C_\lambda := \cap_{i \in \lambda} C_P(N_i/N_{i+1})$ is a closed subgroup of $P$ and, by \cref{bbb}, it follows that $P= \cup_\lambda C_\lambda$. By the Baire Category Theorem, there exist $\lambda \in \Lambda$, an open normal subgroup $M$ of $P$ and $g \in P$ such that $gM \subseteq C_\lambda$. Since $C_\lambda$ is a subgroup of $G,$ this implies that $M$ itself is contained in $C_\lambda$. In particular, we can find an open normal subgroup $N$ of $G$ contained in $K$ such that $M$ centralizes all the $p'$-chief factors of $N$. Given that $O_p(G)=1$, it follows from $M\cap N \leq O_p(N) \leq O_p(G)$ that $M \cap N = 1$ and, consequently, that $P$ is finite. Now we may apply \cref{sylowfg} and reach the desired conclusion. \\
    \indent Now, suppose $G$ is an uncountably based profinite group such that $P_p(g,G)>0$ for every $p$-element of $G$ and suppose, by contradiction, that $G$ is not virtually pro-$p$. For each open normal subgroup $N$ of $G$, denote by $\alpha(N)$ the $p'$-part of $|G/N|$. Note that $G$ is not virtually pro-$p$ if, and only if,  $\sup \alpha(N) = +\infty$. Now, we can choose a countable family $\{N_i\}_{i\in \mathbb N}$ of open normal subgroups such that $\sup_i \alpha(N_i) = +\infty$. Let $M= \cap_i N_i$ and consider $G/M$. On one hand, by construction, $G/M$ is not virtually pro-$p$ and, on the other, it is a countably based profinite group such that $P_p(G/M,gM) > 0$ for every $p$-element $gM$ of $G/M$, a contradiction with the previous paragraph.
\end{proof}

\begin{thm}\label{rand-p-elem}
	Let $P_p(G,G)$ be the probability that two randomly chosen elements of $G$ generate a pro-$p$ group. If $P_p(G,G)$ is positive, then 
	$G$ is virtually pro-$p.$
\end{thm}
\begin{proof}
	Suppose $P_p(G,G)>0.$ By the main result in \cite{winilp}, $G$ contains an open normal subgroup $N$ which is pronilpotent. If $H$ is an open normal pro-$p$ subgroup of $G$, it can be seen that $$P_p(G/H,G/H)=P_p(G,G),$$ because $\langle x,y\rangle$ is a pro-$p$ group if, and only if, $\langle x,y \rangle H/H$ is a pro-$p$ group, so, in particular, it is not restrictive to assume that $O_p(G)=1.$ Let $\Omega$ be the set of the pairs $(x_1N,x_2N)\in G/N\times G/N$ with the property that
	$\langle x_1N, x_2N\rangle$ is a $p$-subgroup of $G/N$
	and for any $\omega=(x_1N,x_2N)\in \Omega,$ let
	$$\Lambda_\omega=\{(x_1n_1,x_2n_2)\in x_1N_1\times x_2N_2\mid \langle x_1n_1, x_2n_2\rangle \text{ is a pro-$p$ subgroup of $G$}\}.$$ Moreover let
	$\Lambda=\{(x_1,x_2)\in G\times G\mid \langle x_1, x_2\rangle \text{ a pro-$p$ subgroup of $G$}\}.$
	
	Since $\Lambda=\cup_{\omega\in \Omega}\Lambda_\omega$ and $\Omega$ is finite, if $P_p(G,G)=\mu(\Lambda)>0,$ then there exists $\omega=(x_1N,x_2N)\in \Omega$ such that $\mu(\Lambda_\omega)>0.$ In particular $\Lambda_\omega\neq \emptyset$ so it is not restrictive to assume $(x_1,x_2)\in \Lambda_\omega.$ Since the order of $N$, as a supernatural number, is coprime with $p$, $(x_1n_1,x_2n_2)\in \Lambda_\omega$ if and only if $\langle x_1n_1,x_2n_2\rangle$ is a $p$-Sylow subgroup of
	$N\langle x_1,x_2\rangle$, and therefore, by the profinite generalization of the Sylow theorems, if and only if $\langle x_1,x_2\rangle=\langle x_1n_1,x_2n_2\rangle^n$ for some $n\in N.$ In other words, it follows that $\Lambda_\omega=\{(x_1^n,x_2^n)\mid n\in N\}$
	and, from $\mu(\Lambda_\omega)>0$, we conclude that $N$ is finite.
\end{proof}

Therefore, it remains simply to combine the results of this section to prove \cref{main}.

\begin{proof}[Proof of \cref{main}.]
    $(1) \Leftrightarrow (2) \Leftrightarrow (3)$ follows from \cref{virtprop} and $(3) \Leftrightarrow (4)$ follows from \cref{rand-p-elem}; it is clear that $(4) \Leftrightarrow (5)$.
\end{proof}

\end{document}